\documentclass[12pt]{article}
\usepackage{amssymb}
\usepackage{mathrsfs}
\usepackage{amsfonts}

\usepackage[top=1in, bottom=1in, left=1.25in, right=1.25in]{geometry}

\usepackage{amsmath,amsthm,amsfonts,amssymb,amscd}
\usepackage{braket}
\usepackage{graphicx}
 \usepackage{color}
\usepackage{ifpdf}

\allowdisplaybreaks[4]
\newtheorem{Definition}{Definition}[section]
\newtheorem {Lemma}{Lemma}[section]
\newtheorem {Theorem} {Theorem}[section]
\newtheorem{Proposition}{Proposition}[section]
\newtheorem {Corollary}{Corollary}[section]

\begin{document}
\baselineskip 16pt

\title{The first few unicyclic and bicyclic hypergraphs with larger spectral radii
\footnote{This work was supported by the Hong Kong Research Grant Council (Grant Nos. PolyU 501212, 501913, 15302114 and 15300715) and NSF of Guangdong Province (Grant No.2014A030310413).}}

\author{Chen Ouyang\footnotemark[3] , Liqun Qi\footnote{Corresponding author. E-mail: liqun.qi@polyu.edu.cn}~\footnotemark[3], Xiying Yuan\footnotemark[4] \\
\footnotesize{\footnotemark[3] Department of Applied Mathematics, The Hong Kong Polytechnic University,} \\ \footnotesize{Hung Hom, Kowloon, Hong Kong}\\
\footnotesize{\footnotemark[4]  Department of Mathematics, Shanghai University, Shanghai 200444, China} }

\date{}
\maketitle

%\newpage

\begin{abstract}
A connected $k$-uniform hypergraph with $n$ vertices and $m$ edges is called $r$-cyclic if $n=m(k-1)-r+1$. For $r=1$ or $2$, the hypergraph is simply called unicyclic or bicyclic. In this paper we investigate hypergraphs that attain larger spectral radii among all simple connected $k$-uniform unicyclic and bicyclic hypergraphs. Specifically, by using some edge operations, the formula on power hypergraph eigenvalues, the weighted incidence matrix and a result on linear unicyclic hypergraphs, we determined the first five hypergraphs with larger spectral radius among all unicyclic hypergraphs and the first three over all bicyclic hypergraphs.
\end{abstract}
\section{Introduction}

In the past decade, the research on spectra of hypergraphs via tensors have  drawn increasingly extensive interest, accompanying with the rapid development of tensor spectral theory since the initial work of Qi~\cite{Qi05} and Lim~\cite{Lim05}.

Given an integer $k\ge 2$, a $k$-uniform hypergraph $H$ refers to a pair $(V,E)$ where $V$ is a non-empty finite set and $E$ is a family of $k$-sets of $V$.  If some element $e\in E$ or $E$ itself is a multi-set, then $H$ is called a multi-hypergraph. Otherwise, we call $H$ a simple hypergraph \cite{Ber76}. In the sequel, $k$-uniform hypergraph is written as $k$-graph for short and all hypergraphs mentioned are simple uniform hypergraphs, unless otherwise stated.

The elements of $V$ and $E$ are called vertices and edges (or hyperedges for $k\ge 3$) of $H$ respectively.
Denote $n=|V|$ and $m=|E|$. Label the vertices by natural numbers $1,\cdots, n$.

The adjacency tensor $\mathcal{A}=\mathcal{A}(H)$ of a $k$-graph $H$ refers to a multi-dimensional array with entries $\mathcal{A}_{i_1\cdots i_k}$ such that
$$\mathcal{A}_{i_1\cdots i_k}=
\begin{cases}
{1\over (k-1)!} &\text{if $\{i_1\cdots i_k\}$ is an edge of $H$,}  \\
\hspace{0.37cm} 0 &\text{otherwise,}
\end{cases}
$$
where each $i_j$ runs from $1$ to $n$ for $j\in [k]$.
The spectrum of $H$ is defined as the multi-set of eigenvalues of the tensor $\mathcal{A}(H)$. One may refer to the definition of tensor eigenvalues introduced by Qi~\cite{Qi05}. The spectral radius of $H$, denoted by $\rho(H)$, is the maximum modulus among all eigenvalues of $\mathcal{A}(H)$.

In spectral theory of hypergraphs, the spectral radius is an index that attracts much attention \cite{CD12, FTPL15,LSQ15,LM14,YSS16}. This may due to the fine properties of its corresponding eigenvector revealed in \cite{CPZ08,FGH13,YY10}, together with its popularity in graph counterpart (See \cite{BS86,CS57,HoSm75,Ni07,Row88} and references therein).

In 2012, Cooper and Dutle~\cite{CD12} systematically studied the eigen properties of the adjacency tensor of a $k$-graph and obtained hypergraph generalizations of many basic results of spectral graph theory.

In 2015, Li, Shao and Qi~\cite{LSQ15}  determined the unique $k$-graph with maximum spectral radius among all supertrees  by studying  perturbations of spectral radius under certain edge operations. The next year, Yuan, Shao and Shan~\cite{YSS16} proceeded to order the uniform supertrees with larger spectral radii by their newly introduced edge operation and a relation established by Zhou et al.~\cite{ZSWB14} between spectral radius of an ordinary graph and its $k$th power.

Recently, Fan, Tan, Peng and Liu~\cite{FTPL15} investigated the hypergraphs that attain largest spectral radii among all unicyclic and bicyclic $k$-graphs. They determined the linear hypergraph with maximum spectral radius over all linear unicyclic $k$-graphs and proposed several candidates for the bicyclic case. Later, Kang et al.~\cite{KLQY16} proved a conjecture in \cite{FTPL15} which lead to the hypergraph maximizing the spectral radius among all linear bicyclic $k$-graphs.

Motivating by the preceding work on maximizing and ordering spectral radius, we take  non-linear $k$-graphs into consideration and try to characterize the first few hypergraphs with larger spectral radii among all unicylic and bicyclic $k$-graphs.

The remainder of this paper is as follows. Section 2 presents relevant notations and some methods useful to later proofs, including the spectral radii perturbations under edge operations, spectra of power hypergraphs from graphs and the construction of weighted incidence matrices in comparing spectral radii. In Section 3, with the application of all these tools, the first five hypergraphs with larger spectral radii among all unicyclic $k$-graphs are determined. The final section further gives the first three hypergraphs that attain larger spectral radii over all bicyclic $k$-graphs. %In the final section, among all $r$-cyclic hypergraphs where $n=m(k-1)-r+1$, the one(s) attaining maximal spectral radius are investigated.

\section{Preliminaries}

Let $H=(V,E)$ be a $k$-graph with $n$ vertices and $m$ edges.
Let $E' \subset E$ and $V'=\cup_{e\in E'}e\subset V$. Then $H'=(V',E')$ is also a $k$-graph and is called a (partial) sub-hypergraph \cite{Ber76}, or simply a subgraph of $H$ (induced by $E'$).

Recall that a path in $H$ refers to an alternative sequence of distinct vertices and edges such that two consecutive vertices are contained in the edge between them in this sequence. If every two vertices in $H$ appear in at least one path, then $H$ is called a connected hypergraph.

%It is proved by Pearson and Zhang~\cite{PZ13} that the adjacency tensor of a connected hypergraph is weakly irreducible.  Then according to the Perron-Frobenius Theorem for weakly irreducible tensors~\cite{FGH13}, there exists a unique positive real vector ${\bf x}=(x_1,\cdots,x_n)^T$ with $\sum_{i=1}^n x_i^k=1$ such that
%
%$$\sum_{i_2,\cdots,i_k=1}^n \mathcal{A}_{ii_2\cdots i_k}x_{i_2}\cdots x_{i_k}=\rho(H)x_i^{k-1}, ~i=1,\cdots,n. $$
%The vector ${\bf x}$ is called the Perron vector or principle eigenvector of $\mathcal{A}$ and $H$. Moreover, ${\bf x}$ and $\rho(H)$ are the optimal solution and optimal value of the maximizing problem below \cite{Qi05}:
%
%$$ \rho(H)= \max\Bigg\{\mathcal{A}x^k:=\sum_{i_1,\cdots,i_k=1}^n \mathcal{A}_{i_1\cdots i_k}x_{i_1}\cdots x_{i_k} : x\in \mathbb{R}^n_+,\sum_{i=1}^n {x_i}^k=1 \Bigg\}. $$

A cycle in $H$ is formed from a path and another edge in $H$ containing the two end vertices of that path. The number of edges in this cycle is called its length.  %The girth of a hypergraph refers to the length of its shortest cycle.
An edge containing in a cycle is called a cycle edge.

A $k$-graph on $n$ vertices and $m$ edges is called $r$-cyclic if $m(k-1)-n+l=r$, where $l$ is the number of its connected components~\cite{FTPL15}.  Note that $r\ge 0$, then for any simple $k$-graph we have $n\le m(k-1)+l$. Moreover, $r=0$ if and only if the uniform hypergraph is acyclic, i.e. it has no cycle \cite[Proposition 4, p.392]{Ber76}. A $1$-cyclic $k$-graph is also called a unicyclic $k$-graph and a bicyclic $k$-graph refers to a $2$-cyclic $k$-graph.

\begin{Lemma}\label{cyclic}
Let $H=(V,E)$ be a simple connected $r$-cyclic $k$-graph with $n$ vertices and $m$ edges.   Let $H_1=(V_1,E_1)$ be a connected subgraph of $H$. If $H_1$ is $r_1$-cyclic, then  $r_1 \le r$.
\end{Lemma}

\begin{proof}
 Let $E_2=E\backslash E_1$, $V_2=\cup_{e\in E_2}e$. Then $H_2=(V_2,E_2)$ is a $k$-uniform subgraph of $H$. Suppose that $|V_i|=n_i$ and $|E_i|=m_i$ for $i=1,2$.
 Since $H_1$ is connected and $r$-cyclic, we have
 $$n_1=m_1(k-1)-r_1+1.$$
Suppose that $H_2$ has $l$ connected components, then $n_2\le m_2(k-1)+l$. Moreover, since $H$ is connected, each component of $H_2$ intersects with $H_1$ at some vertices. Therefore, $n_1+n_2\ge n+l$. Then we have
$$n\le n_1+n_2-l \le m_1(k-1)-r_1+1+m_2(k-1)+l-l=m(k-1)-r_1+1.$$
Thus $r_1\le m(k-1)-n+1=r$.
\end{proof}

Denote by $\mathbb{U}^{m}$ and $\mathbb{B}^{m}$
the set of all connected uniform unicyclic and bicyclic hypergraphs with $m$ edges respectively,  where $m\ge 2$.

\begin{Proposition}\label{prop}
Let $H$ and $F$ be two $k$-graphs in $\mathbb{U}^{m}$ and $\mathbb{B}^{m}$ respectively. Then

$(i)$ ~~every two vertices in $H$ share at most two common edges;

$(ii)$ ~every three vertices in $H$ have at most one common edge;

$(iii)$ every two vertices in $F$ share at most three common edges;

$(iv)$ every three vertices in $F$ have at most two common edges;
\end{Proposition}

\begin{proof}
 If there exist two vertices in $H$ having three common edges, or there are three vertices sharing two common edges, then the subgraph in $H$ induced by those common edges is bicyclic, which contradicts Lemma~\ref{cyclic}.

If $F$ has a pair of vertices sharing four common edges, then there is a $3$-cyclic subgraph in $F$ induced by the four edges, which contradicts Lemma~\ref{cyclic}. If there are three vertices in $F$ sharing three common edges, then the subgraph induced by the three edges is $4$-cyclic, which is a contradiction with Lemma~\ref{cyclic}.
\end{proof}

It is verified in \cite[Lemma~2.1]{FTPL15} that if $H$ contains exactly one cycle, then it is unicyclic ($1$-cyclic). Now  we prove the inverse.

\begin{Lemma}\label{1cyclic}
Let $H$ be a simple connected $k$-graph. Then $H$ is unicyclic $(1$-$cyclic)$ if and only if it has only one cycle.
\end{Lemma}

\begin{proof}
It suffices to prove the necessity. Let $H=(V,E)$ with $|V|=n$, $|E|=m$.

Let $e_1$ be a cycle edge contained in cycle $C=v_1e_1v_2\cdots v_se_s v_1$. Let $w$ be a new vertex and $f=(e_1\backslash \{v_1\})\cup \{w\}$. Then  $H'=(V\cup \{w\}, (E\backslash \{e_1\})\cup \{f\})$ is a connected $k$-graph with $n+1$ vertices and $m$ edges.

Since $H$ is unicyclic, $n=m(k-1)$. Thus $n+1=m(k-1)+1$ which implies that $H'$ is acyclic. Hence all cycles in $H$ contain $e_1$.
According to the arbitrariness of $e_1$, it can be concluded that each cycle edge is contained in every cycle of $H$. In other words, all cycles in $H$ have the same edge set with the same length $s$.

If $s=2$, then by Proposition~\ref{prop} $(ii)$, the two cycle edges intersect at exactly two vertices. Thus $H$ has a unique cycle.

Suppose that $s\ge 3$. Denote by $F$ the subgraph induced by all cycle edges in $H$ on $n'$ vertices. Note that all edges of $F$ can be arranged in a cyclic sequence such that every two consecutive edges share at least one common vertices. If there exists two consecutive edges in $F$ intersecting at two vertices, then $$n'\le s+1+s(k-2)-2.$$
Thus $r'=s(k-1)-n'+1\ge 2$ which implies that $F$ is $r'$-cyclic subgraph with $r'\ge 2$, a contradiction with Lemma~\ref{cyclic}. Therefore, every two consecutive edges in $F$ intersects at only one vertex, which indicates that $H$ has a unique cycle.
 %Let $C=v_1e_1v_2\cdots v_se_s v_1$ be an arbitrary cycle in $H$ where $s\ge 2$. Now consider $C$ as a connected $k$-uniform subgraph of $H$ induced by $E=\{e_i:i\in[s]\}$ with vertex set $V=\cup_{i=1}^s e_i$.
%
%If $s=2$, then by Proposition~\ref{prop} (ii), $|e_1\cap e_2|=2$ and $|V|=s(k-1)$. If $s\ge 3$, then $|e_i\cap e_j|= 1$ for $j-i\equiv 1$ mod $s$. Otherwise there exists two consecutive edges intersecting at two or more vertices, which implies that $|V|\le s+1+(s-2)(k-2)+2(k-3)=s(k-1)-1$. Thus by definition, $C$ is $r$-cyclic with $r\ge 2$, which contradicts Lemma~\ref{cyclic}. Thus $|V|=s(k-1)$ when $s\ge 2$.
%
%Suppose $H$ have another cycle $C'=(V',E')$, then similarly we have $|V'|=|E'|(k-1)$.
%Let $C''=(V\cup V',E\cup E')$. Suppose that $C''$ is $r''$-cyclic.
%
%If $|E\cap E'|=a\ge 1$, then $|V\cup V'|\le (s+|E'|)(k-1)-[(a-1)(k-1)+k]$. Thus
%$r''=|E\cup E'|(k-1)-|V\cup V'|+1 \ge (s+|E'|-a)(k-1)-(s+|E'|)(k-1)+a(k-1)+2=2$, which leads to a contradiction.
%
%Therefore $E\cap E'=\emptyset$. Let $l$ be the number of connected components of $C''$ and $|V\cap V'|=b$. Note that $1\le l\le 2$ and $l+b\ge 2$. Thus
%$$r''=|E\cup E'|(k-1)-|V\cup V'|+l=(s+|E'|)(k-1)-[(s+|E'|)(k-1)-b]+l\ge 2,$$
%which contradicts Lemma~\ref{cyclic}.
%
%Hence $H$ only has one cycle.
\end{proof}

%The above lemma implies that each simple connected $1$-cyclic hypergraph has only one cycle, otherwise it has an $r$-cyclic subgraph that contains at least two cycles with $r\ge 2$.

%This paper mainly focuses on hypergraphs in $\mathbb{U}^{m,k}$ and $\mathbb{B}^{m,k}$.
%and $\mathbb{B}^{m,k}$ and bicyclic hypergraphs respectively

\subsection{Perturbations of spectral radii under edge operations}

%Internal path refers to a sequence of vertices $v_0,\cdots,v_k$ such that those vertices induce ($i$) a cycle or ($ii$) a dumbbell which is a graph obtained by adding at least two pendent edges on each end vertex of a path.
%
%$G_{uv}$ is obtained from the connected graph $G$ by subdividing the edge $uv$, that is by replacing $uv$ with edges $uw$ and $wv$ where $w$ is an additional vertex.
%
%\begin{Lemma}\cite{HoSm75}\label{int}
%(a) If $uv$ lies on an internal path of the connected graph $G$, and $G$ is not a dumbbell, then $\rho(G_{uv})<\rho (G)$;
%
%(b) If $uv$ does not lie on an internal path of  the connected graph $G$, and $G$ is not a cycle, then $\rho(G_{uv})>\rho (G)$.
%\end{Lemma}

In this subsection, we present two edge operations introduced in~\cite{LSQ15} and \cite{YSS16} that help investigating $k$-graphs with larger spectral radii.

Two vertices contained in one edge are called adjacent to each other and said to be connected by this edge. An edge $e$ that contains a vertex $v$ is called an incident edge of $v$. If a vertex has exactly one incident edge, then it is called a pendent vertex, otherwise it is called non-pendent. A pendent edge in a $k$-graph is an edge containing $k-1$ pendent vertices.

\begin{Definition}\cite{LSQ15}\label{em}
Let $r\ge 1$ and let $H=(V,E)$ be a $k$-graph with $u\in V$ and $e_1,\cdots, e_r\in E$ such that $u\notin \cup_{i=1}^re_i$. Suppose that $v_i\in e_i$ and write $e'_i=(e_i\backslash \{v_i\})\cup \{u\}$ for $i\in [r]$. Let $H'=(V,E')$ be the hypergraph with $E'=(E\backslash \{e_i:i\in [r]\})\cup \{e'_i:i\in [r]\}$. Then we say that $H'$ is obtained from $H$ by moving edges $(e_1,\cdots,e_r)$ from $(v_1,\cdots,v_r)$ to $u$.
\end{Definition}

\begin{Lemma}\cite{LSQ15}\label{e0}
Let $r\ge 1$ and let $H$ be a connected $k$-graph. Let $H'$ be the hypergraph obtained from $H$ by moving edges $(e_1,\cdots,e_r)$ from $(v_1,\cdots,v_r)$ to $u$. Assume that $H'$ contains no multiple edges. If ${\bf x}$ is a Perron vector of $H$ and $x_u\ge \max_{1\le i\le r}x_{v_i}$, then $\rho(H')>\rho (H)$.
\end{Lemma}

The following lemma follows directly from Lemma~\ref{e0}.

\begin{Lemma}\label{e}
Let $H$ be a connected $k$-graph and $v_1,\cdots, v_r$ be some of its vertices for $r\ge 2$. Let $H_i$ be a simple hypergraph obtained from $H$ by moving at least one edge from vertices $\{v_j: j\in [r]\backslash \{i\}\}$ to $v_i$. Then we have
$$\max\{\rho(H_i):i\in [r]\}>\rho(H).$$
\end{Lemma}

 From Lemma~\ref{e}, we have the corollary below for a special case.

\begin{Corollary}\label{pm}
Let $H$ be a connected $k$-graph having two adjacent vertices $u_1$ and $u_2$. Let $H'$ be the hypergraph obtained from $H$ by moving all incident edges of $u_2$ except all common edges shared by $u_1,u_2$  from $u_2$ to $u_1$. If $H'\ncong H$, then $$\rho(H) < \rho(H').$$

%Suppose that each edge containing $u_2$ is incident with $u_1$ in $H$. Let $H_{a,b}$ be the hypergraph obtained from $H$ by attaching  $a$ and $b$ pendent edges at $u_1$ and $u_2$ respectively, where $a,b\ge 0$. If $H_{a,b}\ncong H_{a+b,0}$, then
%$$\rho(H_{a,b}) <  \rho(H_{a+b,0}).$$
\end{Corollary}

\begin{proof}
If $u_1$ or $u_2$ does not have other incident edges except their common edges, then $H\cong H'$. Thus $H'\ncong H$ implies that $u_1,u_2$ each has incident edges other than the edges they share. Let $H''$ be the hypergraph  obtained from $H$ by moving all incident edges of $u_1$ except all common edges shared by $u_1,u_2$  from $u_1$ to $u_2$. Note that $H'$ and $H''$ do not have multiple edges since all common edges of $u_1,u_2$ remain unchanged. Moreover, $H''\cong H'$. By Lemma~\ref{e}, $\rho(H) <\max\{\rho(H'),\rho(H'')\}=\rho(H')$.
%If $b=0$, then $H_{a+b,0}\cong H_{a,0}$. If all incident edges of $u_1$ in $H_{a,b}$ are incident with $u_2$, then $a=0$ and $H_{0,b}\cong H_{b,0}\cong H_{a+b,0}$. Thus $H_{a+b,0}\ncong H_{a,b}$ implies that $b\ge 1$ and there exists an edge in $H_{a,b}$ which contains $u_1$ but does not contain $u_2$. Then $H_{a+b,0}$ can be obtained from $H_{a,b}$ by moving $b$ pendent edges from $u_2$ to $u_1$, or by moving all the edges that incident with $u_1$ but not incident with $u_2$ from $u_1$ to $u_2$.   Thus by Lemma~\ref{e0},
%$\rho(H_{a,b}) < \rho(H_{a+b,0})$.
\end{proof}

\begin{Lemma}\cite{YSS16}\label{ne}
Let $k\ge 3$, $H$ be a connected $k$-graph on $n$ vertices having two edges $e$ and $f$ such that $|e\cap f|=k-r$ $(2\le r\le k-1)$. Let $V_1=e\cap f$ and $e\backslash V_1=\{u_1,\cdots,u_r\}$ and $f\backslash  V_1=\{v_1,\cdots,v_r\}$ where $r\ge 2$, $u_1,v_1$ are non-pendent vertices while $u_2,\cdots,u_r$ and $v_2,\cdots,v_r$ are pendent vertices.   Let $H_{e,f}$ be  the hypergraph obtained from $H$ by moving all the edges incident with $v_1$ except $f$ from $v_1$ to $u_2$. Then $\rho(H_{e,f})>\rho (H)$.
\end{Lemma}

\subsection{From graphs to power hypergraphs}

Let $G$ be a graph containing no loops, i.e. cycles of length $1$. The $k$th power of $G$ is defined as the $k$-graph $G^k$ obtained from $G$ by blowing up its edges to hyperedges through adding $k-2$ new pendent vertices to each edge of $G$.

If a hypergraph can be seen as a power of some graph without loops, then it is called a power hypergraph~\cite{HQS13}. %, otherwise it is called non-power.
Observe that a $k$-graph is a power hypergraph if and only if each of its edge contains at least $k-2$ pendent vertices. %Thus a hypergraph having three or more adjacent non-pendent vertices is non-power.

A simple hypergraph is called linear, if each pair of its edges intersects at no more than one vertex~\cite{Bre13}, otherwise it is called non-linear. The powers of a simple graph are always linear, while the $k$th power of a multi-graph is non-linear.

Recall that the adjacency matrix of a multi-graph~\cite{BM08} on $n$ vertices without loops is an $n\times n$ matrix whose $(ij)$-entry is the number of parallel edges connecting $i$ and $j$ if $i\ne j$ and zero otherwise.

Zhou et al.~\cite{ZSWB14} established the following relationship which enables us to acquire spectral information of a power hypergraph from the graph that generates it.

\begin{Lemma}\cite{ZSWB14}\label{p}
If $\lambda\ne 0$ is an eigenvalue of a graph $G$, then $\lambda^{2\over k}$ is an eigenvalue of $G^k$. Moreover, $\rho(G^k)=\rho(G)^{2\over k}$.
\end{Lemma}

{\noindent \bf Remark.} In \cite[Theorem 16]{ZSWB14}, $G$ refers to a simple graph. However, it can be verified through the original proof that  Lemma~\ref{p} also works for multi-graphs without loops.

\vspace{0.3cm}
%\begin{figure}[htb]
%  \begin{center}
%    \includegraphics[width=0.8\linewidth]{figures/graphs.png}
%    \caption{Unicyclic graphs $G_1\thicksim G_7$ with $m$ edges.}
%    \label{g}
%  \end{center}
%\end{figure}
%Some unicyclic graphs relevant to the remaining sections are depicted in Figure~\ref{g}. We may identify their spectral radii through characteristic polynomials.

Denote by $\phi_G(x)=\det(xI-A(G))$ the characteristic polynomial of a graph $G$, where $A(G)$ is the adjacency matrix of $G$ and $I$ denotes the unit matrix.  If $G$ is obtained from two disjoint graphs $H$ and $K$ by amalgamating a vertex $u$ of $H$ and $v$ of $K$,
 then we have the following relation from \cite[Remark 1.6]{Row87}:
\begin{equation}
 \phi_G(x)=\phi_H(x)\phi_{K-v}(x)+\phi_{H-u}(x)\phi_K(x)-x\phi_{H-u}(x)\phi_{K-v}(x), \tag{$*$}
\end{equation}
where $H-u$ and $K-v$ denote the graphs obtained from $H$ and $K$ by deleting $u$ and $v$ and all their incident edges respectively.

Let $G(a,b)$ be a multi-graph obtained from a cycle of length 2 by attaching $a$ and $b$ pendent edges at its two vertices $u$ and $v$ respectively. Denote by $M(a,b)$ the multi-graph obtained from $G(a,b)$ by adding a new edge connecting $u$ and $v$ (See Figure~\ref{g}).

\input{graphs.TpX}

\begin{Lemma}\label{graph}
Let $G_1,G_2,G_3$ and $G(a,b)$ be the unicyclic graphs depicted in Figure~\ref{g} with $m$ edges.  Then for $m\ge 8$,
$$ \rho(G(m-2,0))>\rho(G_3)\ge\rho(G(m-4,2))>\max\{\rho(G_1),\rho(G_2)\}, $$
equality holds only if $m=8$.
\end{Lemma}

\begin{proof}
Since $G_1$ can be obtained from a triangle $C_3$ and a star $K_{1,m-3}$ by amalgamating a vertex of $C_3$ and the unique non-pendent vertex of $K_{1,m-3}$,  by ($*$) we have
\begin{eqnarray*}
\phi_{G_1}(x)
&=& x^{m-3}\cdot \phi_{C_3}(x)+\phi_{P_2}(x)\cdot \phi_{K_{1,m-3}}(x)-x\cdot x^{m-3}\cdot \phi_{P_2}(x)\\
 & = & x^{m-4}(x+1)[x^3-x^2-(m-1)x+m-3],
\end{eqnarray*}
where $P_2$ is a path with one edge.
Similarly by using the amalgamating operation, we obtain the following characteristic polynomials.
\begin{eqnarray*}
%\phi_{G(m-2,0)}(x)
%&=& x^{m-2}\cdot \phi_{C_2}(x)+x\cdot \phi_{K_{1,m-2}}(x)-x\cdot x^{m-2}\cdot x\\
%& = & x^{m-2}(x^2-m-2),\\
%\phi_{G(m-4,2)}(x)
%&=& x^{m-4}[x^4-(m+2)x^2+2(m-4)],\\
\phi_{G(a,b)}(x)
&=& x^{m-4}[x^4-(m+2)x^2+ab],\\
\phi_{G_2}(x)& = & x^{m-4}[x^4-(m+2)x^2+4(m-3)] ,\\
\phi_{G_3}(x)& = & x^{m-4}[x^4-(m+2)x^2+m].% ,\\
%\phi_{G_6}(x)& = & x^{m-6}(x^2-1)[x^4-(m+1)x^2+m-2],\\
%\phi_{G_7}(x)& = & x^{m-4}[x^4-(m+2)x^2+2m-1].
\end{eqnarray*}
Thus
$$\rho(G(m-2,0))^2= m+2,~\rho(G(m-4,2))^2= {1\over2}\left(m+2+\sqrt{m^2-4m+36}\right),$$
$$
\rho(G_2)^2= {1\over2}\left(m+2+\sqrt{m^2-12m+52}\right), ~
    \rho(G_3)^2= {1\over2}\left(m+2+\sqrt{m^2+4}\right).$$

    %,\\
 %&\rho(G_6)^2= {1\over2}(m+1+\sqrt{m^2-2m+9}),  &\quad
   % &\rho(G_7)^2= {1\over2}(m+2+\sqrt{m^2-4m+8})

It is clear that when $m\ge 8$,
$$\rho(G(m-2,0))^2>\rho(G_3)^2 \ge \rho(G(m-4,2))^2>\rho(G_2)^2,$$
 equality holds only if $m=8$. This relationship also holds for the corresponding spectral radii.%\max\{,\rho(G_7)^2\}

 %Let $Y(m)=\rho(G_6)^2-\rho(G_3)^2=\sqrt{m^2-2m+9}-\sqrt{m^2-4m+36}-1$. Since
%$$ {m-1\over m-2}\cdot {\sqrt{m^2-4m+36} \over\sqrt{m^2-2m+9}}={\sqrt{1+32/(m-2)^2}\over \sqrt{1+8/(m-1)^2}}>1, $$
%we have $Y'(m)>0$ for $m>2$. The monotonicity of $Y$ on $m$ together with $\lim_{m\rightarrow +\infty} Y(m)=0$ implies that $Y(m)<0$ for $m>2$. Hence $\rho(G_3)^2>\rho(G_6)^2$ and thus $\rho(G_3)>\rho(G_6)$.

Now it remains to compare $\rho(G(m-4,2))$ and $\rho(G_1)$.

Let $\rho=\rho(G_1)$. Then from $\phi_{G_1}(x)$ we have $\rho^3=\rho^2+(m-1)\rho-m+3$.  Let $g(x)=x^4-(m+2)x^2+2(m-4)$. Then
\begin{eqnarray*}
g(\rho)&=& \rho^4-(m+2)\rho^2+2(m-4)   \\
& = & \rho[\rho^2+(m-1)\rho-m+3]-(m+2)\rho^2+2(m-4)\\
& = & \rho^3-3\rho^2-(m-3)\rho+2(m-4)\\
& = & \rho^2+(m-1)\rho-m+3-3\rho^2-(m-3)\rho+2(m-4)\\
& = & -2\left(\rho-{1\over 2}\right)^2+m-{9\over 2}.
\end{eqnarray*}
Since $\rho> \rho(K_{1,m-1})=\sqrt{m-1}$, we have for $m\ge 6$ that
$$g(\rho)<m-{9\over 2}-2\left(\sqrt{m-1}-{1\over 2}\right)^2<2\sqrt{m-1}-m<0.$$
According to $\phi_{G(a,b)}(x)$, $\rho(G(m-4,2))$ is the largest zero point of $g(x)$, thus it is strictly larger than $\rho=\rho(G_1)$. The proof is completed.
\end{proof}

%According to Lemma relationship,  spectral radii of their $k$th powers can be obtained easily.

\subsection{Weighted incidence matrix in comparing spectral radius }

In \cite{LM14}, Lu and Man introduced the weighted incidence matrix for hypergraphs. They discovered a way to characterize the spectral radius in terms of a particular value $\alpha$ by constructing consistent $\alpha$-normal, $\alpha$-subnormal or $\alpha$-supernormal weighted incidence matrix for the target hypergraph.

\begin{Definition}\cite{LM14}
A weighted incidence matrix $B$ of a hypergraph $H=(V,E)$ is a $|V|\times |E|$ matrix such that for any vertex $v$ and any edge $e$, the entry $B(v,e)>0$ if $v\in e$ and $B(v,e)=0$ if $v \notin e$.
\end{Definition}

\begin{Definition}\cite{LM14}\label{alpha}
A hypergraph $H$ is called $\alpha$-$subnormal$ if there exists a weighted incidence matrix $B$ satisfying

$(a)$ $\sum_{e:v\in e}B(v,e)\le 1$, for any $v\in V(H)$;

$(b)$ $\prod_{v\in e}B(v,e)\ge \alpha$, for any $e\in E(H)$.\\
If no strict inequality appears in $(a)$ and $(b)$, then $H$ is $\alpha$-normal. Otherwise, $H$ is called strictly $\alpha$-subnormal. If furthermore,
$$\prod_{i=1}^l{B(v_i,e_i)\over B(v_{i-1},e_i)}=1$$ for any cycle $v_0e_1v_1e_2\cdots e_lv_0$ $(l\ge 1)$ in $H$, then $B$ is consistent and $H$ is called strictly and consistently $\alpha$-$subnormal$.
\end{Definition}

\begin{Lemma}\cite{LM14}\label{a}
Let $H$ be a $k$-graph. Then

$(i)$ ~$\rho(H)=\alpha^{-{1\over k}}$ if and only if $H$ is consistently $\alpha$-normal;

$(ii)$~if $H$ is strictly and consistently $\alpha$-subnormal, then $$\rho(H)<\alpha^{-{1\over k}}.$$
\end{Lemma}

{\noindent \bf Remark.} In the paper \cite{LM14} of Lu and Man, the spectral radius, say $\rho^*(H)$, is multiplied by a constant factor $(k-1)!$, i.e. $\rho^*(H)=(k-1)!\rho(H)$. Hence we adjust the original formula $\rho^*(H)<(k-1)!\alpha^{-{1\over k}}$ to the above one.

\vspace{0.3cm}
%Given the spectral radius $\rho$ of some graph $G$. If we can find a proper weighted incidence matrix and an $\alpha$ in terms of $\rho$ such that our target
%If we set $\alpha$ in terms of a spectral radius

%The following lemma compares the hypergraphs $U_2(m-4,2)=G_3^k$, $U_3^1(0,0;m-2)$ and $U_3^1(1,0;m-3)$ in this way (See Fig.~\ref{ug} and the detailed definitions in Section~3).
\input{u3.TpX}
 Denote by $U_2(a,b)$ the $k$th power of $G(a,b)$. Let $U_3^1(a,b;c)$ be the $k$-graph obtained from $U_2(a,b)$ by attaching $c$ pendent edges at an arbitrary pendent vertex $w$ in a cycle edge.  Let $U_3^2(a,b;c)$ be the $k$-graph obtained from $U_2(a+1,b)$ by attaching $c$ pendent edges at a pendent vertex $w$ adjacent to $u$ outside the cycle.

The $k$-graphs $U_3^1(a,b;c)$ and $U_3^2(a,b;c)$ are presented in Figure~\ref{ugraph}, where each edge is represented by a closed curve and all non-pendent vertices are in different color.

By letting $\alpha$ be an expression of a certain spectral radius and constructing specific weighted incidence matrices, the following lemma establishes a relation of spectral radii between different hypergraphs.

%\begin{figure}[htb]
%  \begin{center}
%    \includegraphics[width=0.65\linewidth]{figures/hgraphs.png}
%    \caption{Hypergraphs in Lemma~\ref{hgraph}.}
%    \label{hg}
%  \end{center}
%\end{figure}

\begin{Lemma}\label{hgraph}
Let $m\ge 8$. Then for $a\le 1$,
$$\rho(U_3^1(a,0;m-2-a))<\rho(U_2(m-4,2)).$$

\end{Lemma}

\begin{proof}
Let $\alpha=\rho(G(m-4,2))^{-2}$. Since $U_2(m-4,2)$ is the $k$th power of $G(m-4,2)$, by Lemma~\ref{p} we have $\alpha^{-{1\over k}}=\rho(G(m-4,2))^{2\over k}= \rho(U_2(m-4,2))$.

When $m\ge 8$ and $a\le 1$, we claim that $U_3^1(a,0;m-2-a)$ is strictly and consistently $\alpha$-subnormal.

Now we construct a weighted incidence matrix $B$ for $U_3^1(a,0;m-2-a)$. For each pendent vertex $p$ in edge $e$, let $B(p,e)=1$. For non-pendent vertex $q$ in a pendent edge $f$, let $B(q,f)=\alpha$.

 Suppose that $e_1$ and $e_2$ are the two non-pendent edges of $U_3^1(a,0;m-2-a)$ and $w\in e_2$.
Write $x_i=B(u,e_i)$, $y_i=B(v,e_i)$ for $i=1,2$ and $z=B(w,e_2)$. Let
$$x_1+x_2=1-a\alpha,~y_1+y_2=1,~z=1-(m-2-a)\alpha,~x_1y_2=x_2y_1,~x_1y_1=\alpha. $$
Since we have $x_1y_2=x_2y_1$ for the unique cycle $ue_1ve_2u$,  $B$ is consistent according to Definition~\ref{alpha}. It is easy to verify that all equalities hold for ($a$) and ($b$) of Definition~\ref{alpha} except on the edge $e_2$.

Now we compare $x_2y_2z$ with $\alpha$. Let $A={x_2\over x_1}={y_2\over y_1}>0$. Then $$1-a\alpha=(x_1+x_2)(y_1+y_2)=(1+A)^2x_1y_1=(1+A)^2\alpha.$$
Thus $A=\sqrt{{1\over \alpha}-a}-1\ge \sqrt{{1\over \alpha}-1}-1$.
Since $${1\over \alpha}=\rho(G(m-4,2))^2= {1\over2}\left(m+2+\sqrt{m^2-4m+36}\right)> m,$$
when $m\ge 15$ we have that
\begin{eqnarray*}
{x_2y_2z\over \alpha}&=&[1-(m-2-a)\alpha]A^2 \\
& \ge & \left[{1\over\alpha}-(m-2)\right]\left(\sqrt{1-\alpha}-\sqrt{\alpha}\right)^2\\
& > &  2\left(\sqrt{1-{1\over m}}-\sqrt{1\over m}\right)^2\\
& \ge & 2\left(\sqrt{14\over 15}-\sqrt{1\over 15}\right)^2 > 1 .
\end{eqnarray*}

Direct computation shows that the value of $\left[{1\over\alpha}-(m-2)\right](\sqrt{1-\alpha}-\sqrt{\alpha})^2$ rests in the interval $(1.1, 1.4)$ when $8\le m\le 14$. Hence $\prod_{t\in e_2} B(t,e_2)=x_2y_2z>\alpha$ for $m\ge 8$. Thus $U_3^1(a,0;m-2-a)$ is strictly $\alpha$-subnormal by Definition~\ref{alpha}.
By Lemma~\ref{a} $(ii)$, we have
$$\rho(U_3^1(a,0;m-2-a))<\alpha^{-{1\over k}}= \rho(U_2(m-4,2)).$$
\end{proof}

\input{b3.TpX}

Let $B_2(a,b)$ be the $k$th power of $M(a,b)$  depicted in Figure~\ref{g}.
Denote by $B_3^1(a,b,c)$ the $k$-graph with merely two non-pendent edges which intersect at exactly three vertices $u,v,w$, where $a,b,c$ are the number of pendent edges attached at $u,v,w$ respectively. Let $B_3^2(a,b,c)$ ($B_3^3(a,b,c)$, resp.) be the hypergraph obtained from $U_3^1(a,b;c)$ by adding a new edge containing $u,v$ ($v,w$ resp.) and $k-2$ new pendent vertices. Let $B_3^4(a,b,c)$ ($B_3^5(a,b,c)$ and $B_3^6(a,b,c)$ resp.) be the hypergraph obtained from $U_3^2(a,b;c)$ by adding a new edge containing $u,v$ ($u,w$ and $v,w$ resp.) and $k-2$ new pendent vertices. Let $B_4$ be the bicyclic hypergraph obtained from $B_3^1(0,0,0)$ by attaching $m-2$ pendent edges at an arbitrary pendent vertex $t$ in a cycle edge.

\begin{Lemma}\label{bgraph}
Let $m\ge 5$. Then
$$(i)~~\rho(B_3^1(m-2,0,0))=\rho(B_2(m-3,0));\hspace{5.8cm}$$
$$(ii)~ \max\{\rho(B_3^1(m-3,1,0)),\rho(B_3^3(0 ,m-3,0)),\rho(B_4)\}<\rho(B_2(m-4,1)).$$
\end{Lemma}

\begin{proof}
By using the amalgamating operation and the formula ($*$), we obtain the following characteristic polynomial:
$$\phi_{M(a,b)}(x)
= x^{m-5}[x^4-(m+6)x^2+ab],$$
%\begin{eqnarray*}
%\phi_{M(m-3,0)}(x)
%&=& x^{m-3}(x^2-m-6),\\
% \phi_{M(m-4,1)}(x)
%&=& x^{m-5}[x^4-(m+6)x^2+m-4].
%\end{eqnarray*}
where $a+b=m-3$. Thus $$\rho(M(m-3,0))^2=m+6,~ \rho(M(m-4,1))^2={1\over 2}\left(m+6+\sqrt{m^2+8m+52}\right).$$

Let
$$\alpha=\rho(M(m-3,0))^{-2}={1\over m+6}, ~~\beta=\rho(M(m-4,1))^{-2}.$$

{\bf Claim 1.} $B_3^1(m-2,0,0)$ is consistently $\alpha$-normal.

Now we construct a weighted incidence matrix $B$ for $B_3^1(m-2 ,0,0)$. Let $B(p,e)=1$ for every pendent vertex $p$ in edge $e$ and let $B(q,f)=\alpha$ for each non-pendent vertex $q$ in a pendent edge $f$. Suppose that $e_1$ and $e_2$ are the two edges intersecting at $u,v,w$. Let $B(u,e_i)={1-(m-2)\alpha \over 2}$ and $B(v,e_i)=B(w,e_i)={1\over 2}$ for $i=1,2$.

It can be verified that  $\sum_{e:t\in e}B(t,e)= 1$ for any vertex $t$ and $\prod_{t\in e}B(t,e)= \alpha$ for any edge $e$ in $B_3^1(m-2 ,0,0)$. Moreover, $B$ is consistent for all three cycles in $B_3^1(m-2 ,0,0)$. Therefore by Definition~\ref{alpha}, $B_3^1(m-2,0,0)$ is consistently $\alpha$-normal.

Thus by Lemmas~\ref{p} and \ref{a}~$(i)$,
$$\rho(B_3^1(m-2,0,0))=\alpha^{-{1\over k}}=\rho(M_2(m-3,0))^{2\over k}= \rho(B_2(m-3,0)).$$

{\bf Claim 2.} $B_3^1(m-3,1,0)$ is strictly and consistently $\beta$-subnormal.

We first construct a weighted incidence matrix $B$ for $B_3^1(m-3 ,1,0)$. Let $B(p,e)=1$ for every pendent vertex $p$ in edge $e$ and let $B(q,f)=\beta$ for each non-pendent vertex $q$ contained in a pendent edge $f$ . Suppose that $e_1$ and $e_2$ are the two non-pendent edges.

Write $x_i=B(u,e_i)$, $y_i=B(v,e_i)$ and $z_i=B(w,e_i)$ for $i=1,2$. Let
$$x_1+x_2=1-(m-3)\beta,\quad y_1+y_2=1-\beta,\quad z_1+z_2=1,\quad x_2y_2z_2=\beta,$$
%$$x_1y_2=x_2y_1,\quad x_1z_2=x_2z_1,\quad y_1z_2=y_2z_1,$$
and let $A={x_1\over x_2}={y_1\over y_2}={z_1\over z_2}>0$.

Since $x_1y_2=x_2y_1$, $x_1z_2=x_2z_1$ and $y_1z_2=y_2z_1$ for all three cycles,  $B$ is consistent according to Definition~\ref{alpha}. It is easy to verify that all equalities hold for ($a$) and ($b$) of Definition~\ref{alpha} except on the edge $e_1$.

Now we compare $x_1y_1z_1$ with $\beta$. Note that $$(1-\beta)[1-(m-3)\beta]=(x_1+x_2)(y_1+y_2)(z_1+z_2)=(1+A)^3x_2y_2z_2=(1+A)^3\beta.$$
Since $\beta^{-{1\over 2}}=\rho(M(m-4,1))$ is the largest root of $x^4-(m+6)x^2+m-4=0$, we have that $\beta^{-2}-(m-2)\beta^{-1}+(m-3)=8\beta^{-1}+1$ and thus
\begin{eqnarray*}
(1+A)^3&=&{\beta^{-1}}(1-\beta)[1-(m-3)\beta]\\
&=& \beta[\beta^{-2}-(m-2)\beta^{-1}+(m-3)]\\
&=&\beta(8\beta^{-1}+1)> 8.
\end{eqnarray*}
Therefore $A>1$ and thus
$$\prod_{t\in e_2} B(t,e_2)=x_1y_1z_1=A^3x_2y_2z_2=A^3\beta>\beta.$$
 Hence $B_3^1(m-3,1,0)$ is strictly and consistently $\beta$-subnormal.

%By Lemmas~\ref{p} and \ref{a},
% $$\rho(B_3^1(m-3,1,0)) <  \beta^{-{1\over k}}=\rho(M(m-4,1))^{2\over k}= \rho(B_2(m-4,1)).$$
%
%{\bf Claim 3.} $B_3^2(0,0,m-3)$ is strictly and consistently $\beta$-subnormal.
%
%Now we construct a weighted incidence matrix $B$ for $B_3^2(0,0, m-3 )$.
%Let $e_1,e_2,e_3$ be the three non-pendent edges where $e_3$ contains all of $u,v,w$.
%Write $x_i=B(u,e_i)$, $y_i=B(v,e_i)$ for $i=1,2,3$ and $z=B(w,e_3)$.
%Let
%$$x_1=x_2=y_1=y_2=\beta^{1\over 2}, \quad x_3=y_3=1-2\beta^{1\over 2}, \quad z=1-(m-3)\beta.$$
%Assign $1$ to $B(p,e)$ for every pendent vertex $p$ in edge $e$ and $\beta$ to $B(q,f)$ for each pendent edge $f$ containing a non-pendent vertex $q$.
%
%It is easy to check that $B$ is consistent and all equalities hold for ($a$) and ($b$) of Definition~\ref{alpha} except on the edge $e_3$.
%
% It remains to compare $x_3y_3z$ with $\beta$. If $m\ge 5$, then
% $$\beta=\rho(M(m-4,1))^{-2}={2\over {m+6+\sqrt{(m+4)^2+36}}}<{1\over m+5}<{1\over 10}.$$
% Thus
% $${x_3y_3z\over \beta}=\left(1-2\beta^{1\over 2}\right)^2\left[{1\over\beta}-(m-3)\right]>8\left(1-2\beta^{1\over 2}\right)^2\ge 8\left(1-{2\over \sqrt{10}}\right)^2> 1.$$
%
%Since $\prod_{t\in e_3} B(t,e_3)=x_3y_3z>\beta$,
% $B_3^2(0,0 ,m-3)$ is strictly and consistently $\beta$-subnormal.

 %By Lemmas~\ref{p} and \ref{a},
% $$\rho(B_3^2(0,0 ,m-3)) <  \beta^{-{1\over k}}=\rho(M(m-4,1))^{2\over k}= \rho(B_2(m-4,1)).$$

{\bf Claim 3.} $B_3^3(0,m-3,0)$ is strictly and consistently $\beta$-subnormal.

We construct a weighted incidence matrix $B$ for $B_3^3(0,m-3,0)$.
Let $e_1,e_2,e_3$ be the three non-pendent edges where $\{u,v\} \subset e_2$, $\{v,w\} \subset e_3$ and $e_1$ contains all of $u,v,w$.

Write $x_1=B(u,e_1),x_2=B(u,e_2),z_1=B(w,e_1),z_3=B(w,e_3)$ and $y_i=B(v,e_i)$ for $i=1,2,3$.
Let $A={x_1\over x_2}={y_1\over y_2}={y_1\over y_3}={z_1\over z_3}>0$,
$$x_2= z_3={1\over A+1},\quad y_2=y_3={1-(m-3)\beta \over A+2}, \quad x_2y_2=z_3y_3=\beta.$$

Assign $1$ to $B(p,e)$ for every pendent vertex $p$ in edge $e$ and $\beta$ to $B(q,f)$ for each non-pendent vertex $q$  in a pendent edge $f$.

By the above equalities, $B$ is consistent for all three cycles $ue_1ve_2u$, $ve_1we_3v$ and $ue_1we_3ve_2u$. Moreover, all equalities hold for ($a$) and ($b$) of Definition~\ref{alpha} except on the edge $e_1$.

Since ${1\over \beta}=\rho(M(m-4,1))^2>m+5$, we have
$$(A+1)(A+2)={1-(m-3)\beta \over x_2y_2}={1-(m-3)\beta \over \beta}>8.$$
Thus $A>{\sqrt{33}-3\over 2}>1.37$.
Therefore,
$${x_1y_1z_1\over \beta}={A^3x_2y_2z_3\over \beta} ={A^3\over 1+A}={1\over A^{-3} +A^{-2}}>1.$$

As $\prod_{t\in e_1} B(t,e_1)=x_1y_1z_1>\beta$,
 $B_3^3(0,m-3,0)$ is strictly and consistently $\beta$-subnormal.

%By Lemmas~\ref{p} and \ref{a},
% $$\rho(B_3^3(0, m-3 ,0 )) <  \beta^{-{1\over k}}=\rho(M(m-4,1))^{2\over k}= \rho(B_2(m-4,1)).$$
%
{\bf Claim 4.} $B_4$ is strictly and consistently $\beta$-subnormal.

We construct a weighted incidence matrix $B$ for $B_4$. Let $e_1,e_2$ be the two non-pendent edges in $B_4$ where $t\in e_2$.
Assign $1$ to $B(p,e)$ for every pendent vertex $p$ in edge $e$ and $\beta$ to $B(q,f)$ for each non-pendent vertex $q$ in a pendent edge $f$. Let
$${B(u,e_2)\over B(u,e_1)}={B(v,e_2)\over B(v,e_1)}={B(w,e_2)\over B(w,e_1)}=A,\quad B(t,e_2)=1-(m-2)\beta,$$  $$B(u,e_1)=B(v,e_1)=B(w,e_1)={1\over A+1}=\beta^{1\over 3}.$$

It is easy to check that $B$ is consistent and all equalities hold for ($a$) and ($b$) of Definition~\ref{alpha} except on the edge $e_2$. Since $A=\sqrt[3]{1\over \beta}-1>1$ and $${1\over \beta}=\rho(M(m-4,1))^2>m+5\ge 10$$ when $m\ge 5$,  we have
\begin{eqnarray*}
{1\over \beta}\prod_{s\in e_2} B(s,e_2) &=&[1-(m-2)\beta]A^3\\
&=&\left({1\over\beta}-(m-2)\right)(1-\sqrt[3]{\beta})^3\\
 &>&  7\left(1-\sqrt[3]{1\over 10}\right)^3 > 1.
\end{eqnarray*}
Thus $B_4$ is strictly and consistently $\beta$-subnormal.

By Claims $2,3,4$,  Lemmas~\ref{p} and \ref{a} $(ii)$, we have
 $$\rho(H)< \beta^{-{1\over k}}=\rho(M(m-4,1))^{2\over k}= \rho(B_2(m-4,1))$$
 for  $H\in\{ B_3^1(m-3,1,0),B_3^3(0 ,m-3,0), B_4\}$.
\end{proof}

\section{The first five unicyclic $k$-graphs with larger spectral radii in $\mathbb{U}^{m}$}

We classify $\mathbb{U}^{m}$ by the number of non-pendent vertices. Denote by $\mathbb{U}_i^{m}$ the set of hypergraphs in $\mathbb{U}^{m}$ with exactly $i$ non-pendent vertices. Since the least possible length of a cycle in a simple hypergraph is two, we have $i\ge 2$. %Applying tools listed in Section 2, hypergraphs with larger spectral radii within each class are investigated.

 Note that the $k$th power of $G_1$ in Figure~\ref{g} is in $\mathbb{U}_3^{m}$. In \cite{FTPL15}, $G_1^k$ has been proved to uniquely attain the largest spectral radius over all linear $k$-graphs in $\mathbb{U}^{m}$. Hence the following will be focused on non-linear $k$-graphs.

Let $H$ be a non-linear $k$-graph in $\mathbb{U}_i^{m}$. We claim that the length of the unique cycle of $H$ is $2$. If the length is at least 3, then by Lemma~\ref{1cyclic} $H$ can not have two vertices sharing two common edges which forms another cycle. Hence $H$ is linear which yields a contradiction. Let $ue_1ve_2u$ be the unique cycle in $H$.

  First we consider that $H$ is in $\mathbb{U}_2^{m}$. By Proposition~\ref{prop} $(i)$, $e_1,e_2$ are the only two non-pendent edges that contains $u$ and $v$. As the remaining edges (if there exists any) are pendent, we have $H\cong U_2(a,b)$ for some nonnegative integers $a$ and $b$. Thus $k$-graphs in $\mathbb{U}_2^{m}$ are in the form of $U_2(a,b)$ with $a,b\in \mathbb{N}$.

\begin{Lemma}\label{u2}
Let $a\ge b \ge 1$  and $a+b=m-2$. Then
$$ \rho(U_2(a,b))< \rho(U_2(a+1,b-1))\le \rho(U_2(m-2,0)).$$
\end{Lemma}

\begin{proof}
Note that $U_2(a+1,b-1)$ can be obtained from $U_2(a,b)$ by
 moving one pendent edge from $v$ to $u$, or by moving $a-b+1$ pendent edges from $u$ to $v$,  by Lemma~\ref{e} we have $\rho(U_2(a,b))<\rho(U_2(a+1,b-1))$. By induction, $\rho(U_2(a+1,b-1))\le \rho(U_2(m-2,0))$ with equality if and only if $b=1$.
\end{proof}

%Let $G(a,b)$ be a multi-graph obtained from two stars $S_{a+1}$ and $S_{b+1}$ by connecting their centers $u$ and $v$ with two parallel edges. Then for each $a$ and $b$, $U_2(a,b)$ can be seen as the $k$th power of the multi-graph $G(a,b)$.

%However, $\mathbb{U}_3^{m,k}$ contains both power and non-power hypergraphs.

Now we consider $H\in \mathbb{U}_3^{m}$. Let $w$ be the remaining non-pendent vertex of $H$.  %Let $C_{2,k}$ be a $k$-uniform cycle of length 2 with $u,v$ be the intersection of its two edges $e_1$ and $ e_2$.
If $w$ is in a cycle edge say $e_1$, then by Proposition~\ref{prop} $(ii)$, $w\notin e_2$. Thus $H\cong U_3^1(a,b;c)$ for some integers $a,b$ and $c\ge 1$. If $w$ is not on the cycle, then there is an edge outside the cycle containing $w$ and one of $u,v$, say $u$. Thus $H\cong U_3^2(a,b;c)$ for some integers $a,b$ and $c\ge 1$. Therefore, non-linear $k$-graphs in $\mathbb{U}_3^{m}$ are either in the form of $U_3^1(a,b;c)$ or $U_3^2(a,b;c)$ with $c\ge 1$.

%Let $U$ be a non-linear hypergraph in $\mathbb{U}_3^{m,k}$ with non-pendent vertices $u,v,w$.  Suppose that $ue_1ve_2u$ is the cycle of $U$.  Let ... ...

\begin{Lemma}\label{u3}
Let $H$ be a non-linear $k$-graph in $\mathbb{U}_3^{m}\backslash \{U_3^1(m-3,0;1), U_3^2(m-4,0;1)\}$. If $m\ge 8$, then
$$ \rho(H) < \rho(U_2(m-4,2)) \le \rho(U_3^2(m-4,0;1))< \rho(U_3^1(m-3,0;1)), $$
equality holds only if $m=8$.
\end{Lemma}

\begin{proof}
%If $H$ is linear, then according to \cite[Corollary 3.7]{FTPL15}, Lemmas~\ref{p} and \ref{graph} we have
%$$ \rho(H)\le \rho(G_1^{k})<\rho(G^k(m-4,2))=\rho(U_2(m-4,2)).
%$$
We discuss by two cases.

{\bf Case 1.}  $H \cong U_3^1(a,b;c)$.

Suppose that $a\ge b$.
Since $H\ncong U_3^1(m-3,0;1)$, $b\ge 1$ or $c\ge 2$.

 If $a\ge 2$, then by Corollary~\ref{pm}, take $u_1=v$ and $u_2=w$, we have $$\rho(H) < \rho(U_2(a,b+c)) \le \rho(U_2(m-4,2)).$$ The second inequality follows from Lemma~\ref{u2} for $a\ge 2$ and $b+c\ge 2$.

If $a=b=1$, then by Corollary~\ref{pm}, take $u_1=u$ and $u_2=v$, we have that $\rho(H) < \rho(U_3^1(2,0;c))$ and $c=m-4>2$, which can be ascribed to $a\ge 2$.

If $a\le 1$ and $b=0$, then by Lemma~\ref{hgraph}, $\rho(H) < \rho(U_2(m-4,2))$.

{\bf Case 2.} $H \cong U_3^2(a,b;c)$.

Since $H\ncong U_3^2(m-4,0;1)$, $b\ge 1$ or $c\ge 2$.
Then we can obtain $U_3^1(a+1,b;c)$ from $H$ by moving $c$ pendent edges from $w$ to an arbitrary pendent vertex in a cycle edge. This operation coincides with the one of Lemma~\ref{ne}. Thus for $b \ge 1$ or $c\ge 2$, $$\rho(H)< \rho(U_3^1(a+1,b;c))< \rho(U_2(m-4,2)),$$ where the second inequality follows from  Case $1$.

Since $U_2(m-4,2)$ and $U_3^2(m-4,0;1)$ are the $k$th powers of $G(m-4,2)$ and $G_3$ respectively (See Figure~\ref{g}),  by Lemmas~\ref{p} and \ref{graph},
$$\rho(U_2(m-4,2))=\rho(G(m-4,2))^{2\over k} \le \rho(G_3)^{2\over k}=\rho(U_3^2(m-4,0;1)),$$
and equality holds only if $m=8$.

It remains to prove the last inequality of this lemma. Observe that $U_3^1(m-3,0;1)$ can be obtained from $U_3^2(m-4,0;1)$ by moving the pendent edge attached at $w$ from $w$ to an arbitrary pendent vertex in a cycle edge. By Lemma~\ref{ne}, $$\rho(U_3^2(m-4,0;1))< \rho(U_3^1(m-3,0;1)).$$

The proof is completed.
\end{proof}

Next we discuss $H\in  \mathbb{U}_4^{m}$.

%Suppose that $H$ is a multi-power hypergraph with root graph $G_{H}$. Recall that the cycle in $H$ is of length $2$. Thus $w,t$ are not on the cycle, otherwise there exists an edge containing three non-pendent vertices.
%If $w,t$ are adjacent, then they form a path together with a vertex on the cycle, say $u$, and thus $G_H\cong U_4^1$ (See Fig.~\ref{root}).
%If $w,t$ are not adjacent, then $G_H \cong U_4^2$ if they adjacent to distinct non-pendent vertices on the cycle, and $G_{H}\cong U_4^3$ otherwise. Hence the multi-power hypergraphs in $\mathbb{U}_4^{m,k}$ are partitioned into three kinds with root graphs $U_4^1$, $U_4^2$ and $U_4^3$, respectively.

\begin{Lemma}\label{u4}
Let $H$ be a non-linear $k$-graph in $\mathbb{U}_4^{m}$ with $m\ge8$. Then $$\rho(H)<\rho (U_2(m-4,2)).$$
\end{Lemma}

\begin{proof}
%If $H$ is linear, then according to \cite[Corollary 3.7]{FTPL15} and Lemma~\ref{u3} we have
%$ \rho(H)\le \rho(G_1^k)<\rho(U_2(m-4,2))$.
For $H$ being non-linear, recall that the unique cycle is denoted by $ue_1ve_2u$ with length $2$. Let $w$ and $t$ be the remaining two non-pendent vertices of $H$ and let $a,b,c,d$ be the number of pendent edges attached at $u,v,w,t$, respectively.

We discuss by the location of $w$ and $t$.

{\bf Case 1.} $w,t$ are on the cycle. By Lemma~\ref{1cyclic}, each of $w,t$ is contained in only one non-pendent edge. Thus $c\ge 1, d\ge 1$. By moving all pendent edges from $w$ to $t$ or from $t$ to $w$, we can both obtain $U_3^1(a,b;c+d)$ with $c+d\ge 2$. Then by Lemmas~\ref{e} and \ref{u3}, $\rho(H)<\rho(U_3^1(a,b;c+d))<\rho(U_2(m-4,2))$.

{\bf Case 2.} Only one of $w,t$ is on the cycle, say $w$. Then $d\ge 1$, otherwise $t$ is a pendent vertex or $H$ is not unicyclic.

{\bf Subcase 2.1.}
 $w$ and $t$
are connected by an edge $f$. Take $u_1=w$ and $u_2=t$,  by Corollary~\ref{pm} and Lemma~\ref{u3} we have  $\rho(H)<\rho(U_3^1(a,b;c+d+1))<\rho(U_2(m-4,2))$ for $c+d\ge 1$.

{\bf Subcase 2.2.}
 $w$ and $t$ are not adjacent. Then $t$ is adjacent to $u$ or $v$. Suppose that $w\in e_1\backslash e_2$. Moving $d$ pendent edges from $t$ to an arbitrary pendent vertex in $e_2$, by Lemma~\ref{ne} we obtain a hypergraph $H'$ of Case $1$ with larger spectral radius. Thus $\rho(H)<\rho(H')<\rho(U_2(m-4,2))$.

{\bf Case 3.} $w$ and $t$ are outside the cycle.
Then at least one of $w$ and $t$, say $w$, is adjacent to a vertex on the cycle, say $u$. Suppose that $u,w$ are connected by an edge $f'$ outside the cycle. Moving all edges incident with $w$ expect $f'$ from $w$ to an arbitrary pendent vertex on the cycle, we obtain a hypergraph $H''$ of Case $2$. By Lemma~\ref{ne} and the discussion in Case 2,
 $\rho(H)<\rho(H'')<\rho(U_2(m-4,2)).$

This completes the proof.
\end{proof}

\begin{Lemma}\label{ui}
Let $i\ge 3$ and $H$ be a $k$-graph in $\mathbb{U}_i^{m}$. Then
$$\rho(H)< \max\{\rho(F): \text{$F$ is a $k$-graph in~} \mathbb{U}_{i-1}^{m}\}.$$
\end{Lemma}

\begin{proof}
First we consider that all non-pendent vertices of $H$ are in the same edge say $f$. Then there exists a non-pendent vertex $w$ whose incident edges except $f$ are pendent edges. Otherwise, each non-pendent vertex is incident with at least two non-pendent edges and we can find two distinct cycles with length 2 when $i\ge 3$, which contradicts Lemma~\ref{1cyclic}.
Now move all pendent edges attached at $w$ from $w$ to another non-pendent vertex $t$ in $f$, we obtain a $k$-graph in $\mathbb{U}_{i-1}^{m}$, denoted $H'$. By Corollary~\ref{pm}, take $u_1=t$ and $u_2=w$, we have $\rho(H)<\rho(H').$

Now suppose $H$ have two non-pendent vertices $u,v$ that do not share any common edge. Let $P=ue_1\cdots e_sv$ be a shortest path connecting $u$ and $v$ in $H$ where $s\ge 2$. Let $H_1$ be the $k$-graph obtained from $H$ by moving all edges incident with $u$ except $e_1$ from $u$ to $v$. Let $H_2$ be the $k$-graph obtained from $H$ by moving all edges incident with $v$ except $e_s$ from $v$ to $u$. Note that $H_1$ and $H_2$ are in $\mathbb{U}_{i-1}^{m}$. By Lemma~\ref{e}, $\rho(H) < \max\{\rho(H_1),\rho(H_2)\}.$

For both cases, $\rho(H)$ is bounded up by the spectral radius of a $k$-graph in $\mathbb{U}_{i-1}^{m}$, thus the proof is completed.
\end{proof}

By Lemmas~\ref{u4} and \ref{ui}, we have:

\begin{Lemma}\label{u5}
Let $H$ be a non-linear $k$-graph in $\mathbb{U}_i^{m}$, where $i\ge 4$ and $m\ge 8$. Then

$$\rho(H)<\rho (U_2(m-4,2)).$$
\end{Lemma}

Now a main result of this paper follows.

\begin{Theorem}\label{unicyclic}
Let $H$ be a $k$-graph in $\mathbb{U}^{m}$ with $m\ge 8$. Then
\begin{eqnarray*}
 (i)~\rho(U_2(m-4,2))&\le&\rho(U_3^2(m-4,0;1))<\rho(U_3^1(m-3,0;1))~~~~~\qquad \qquad  \\
&<&\rho(U_2(m-3,1))<\rho(U_2(m-2,0)),
\end{eqnarray*}
equality holds only if $m=8$.

$(ii)$ If $H \notin \{U_2(m-2,0),U_2(m-3,1),U_2(m-4,2),U_3^1(m-3,0;1),U_3^2(m-4,0;1)\}$,
then
$$\rho(H)<\rho(U_2(m-4,2)).$$

\end{Theorem}

\begin{proof}
We first prove the relationship in $(i)$.

The first two inequalities result directly from Lemma~\ref{u3} and the fourth inequality follows from Lemma~\ref{u2}.

For $U_3^1(m-3,0;1)$, by Corollary~\ref{pm}, take $u_1=v$ and $u_2=w$, we obtain the third inequality that $\rho(U_3^1(m-3,0;1))<\rho(U_2(m-3,1))$.

If $H$ is non-linear, the inequality of $(ii)$ can be obtained from Lemmas~\ref{u2}, \ref{u3} and \ref{u5} by specifying the number of non-pendent vertices in $H$. If $H$ is linear, then according to \cite[Corollary 3.7]{FTPL15}, Lemmas~\ref{p} and \ref{graph} we have
$$ \rho(H)\le \rho(G_1^{k})<\rho(G^k(m-4,2))=\rho(U_2(m-4,2)).
$$

The proof is completed. \end{proof}

\section{The first three bicyclic $k$-graphs with larger spectral radii in $\mathbb{B}^{m}$}

Denote by $\mathbb{B}_i^{m}$ the set of hypergraphs in $\mathbb{B}^{m}$ with exactly $i$ non-pendent vertices where $i\ge 2$.
Let $H$ be a $k$-graph in $\mathbb{B}_i^{m}$.

We first consider that $H$ is in $\mathbb{B}_2^{m}$. Let $u,v$ be the non-pendent vertices in $H$. Since $H$ is bicyclic, $u,v$ have at least three common edges, otherwise $H$ is acyclic or unicyclic. By Proposition~\ref{prop} $(iv)$, there are exactly three edges sharing $u,v$. As the remaining edges of $H$ (if there exists any) are pendent edges, $H\cong B_2(a,b)$ for some integers $a,b$. Thus $k$-graphs in $\mathbb{B}_2^{m}$ are in the form of $B_2(a,b)$ with $a,b \in \mathbb{N}$.

\begin{Lemma}\label{b2}
Let $a\ge b \ge 1$ and $a+b=m-3$. Then
$$ \rho(B_2(a,b))< \rho(B_2(a+1,b-1))\le \rho(B_2(m-3,0)).$$
\end{Lemma}
\begin{proof}
Note that $B_2(a+1,b-1)$ can be obtained from $B_2(a,b)$
 by moving one pendent edge from $v$ to $u$, or by moving $a-b+1$ pendent edges from $u$ to $v$. By Lemma~\ref{e},  $\rho(B_2(a,b))<\rho(B_2(a+1,b-1))$. Then by induction, $\rho(B_2(a+1,b-1))\le \rho(B_2(m-3,0))$ with equality if and only if $b=1$.
\end{proof}

Now we investigate $H\in \mathbb{B}_3^{m}$. Let $u,v,w$ be the three non-pendent vertices of $H$. We may discuss by the number of common edges $u,v,w$ have. By Proposition~\ref{prop} $(iii)$, $u,v,w$ share at most two common edges.

If $u,v,w$ have exactly two common edges, then any two of them can not share anther edge, otherwise there is a $3$-cyclic subgraph induced by three non-pendent edges in $H$, which contradicts Lemma~\ref{cyclic}. Hence the remaining edges are pendent edges attached at $u,v$ or $w$. Thus $H\cong B_3^1(a,b,c)$ for some integers $a,b$ and $c$.

If $u,v,w$ have only one common edge say $e_1$, then there are at least two more edges that each contains two non-pendent vertices. Otherwise $H$ is acyclic or unicyclic. If there are three more non-pendent edges other than $e_1$, then $H$ has a $3$-cyclic subgraph formed by four non-pendent edges, which contradicts Lemma~\ref{cyclic}. Hence $H$ has exactly two more non-pendent edges, say $e_2,e_3$. If $e_2, e_3$ intersect at two vertices say $u,v$, then $H\cong B_3^2(a,b,c)$. Otherwise $e_2$ and $e_3$ have only one common vertex say $v$, then $H\cong B_3^3(a,b,c)$.

Suppose that $u,v,w$ do not have common edge.
Since $H$ is connected, there is a path connecting $u,v,w$, say $ve_1ue_2w$.
As $H$ is bicyclic, there are exactly two more non-pendent edges, say $e_3$ and $e_4$, that each contains two of $u,v,w$. Otherwise $H$ is acyclic, unicyclic or has a $3$-cyclic subgraph formed by five non-pendent edges. Note that $e_3\cap e_4< 3$. If $e_3\cap e_4$ is $\{u,v\}$ or $\{u,w\}$, then $H\cong B_3^4(a,b,c)$ for some $a,b,c$. If  $e_3\cap e_4=\{u\}$, then $H\cong B_3^5(a,b,c)$. Otherwise $e_3\cap e_4$ is $\{v,w\}$, $\{v\}$ or $\{w\}$, then $H\cong B_3^6(a,b,c)$ for some $a,b,c$.

%Since a linear bicyclic hypergraph has at least four non-pendent vertices~\cite{FTPL15}, $H$ is non-linear. Thus we can always find two vertices in $H$ sharing at least two edges. If there are two non-pendent vertices, say $u,v$, sharing three common edges, then $H\cong B_3^4(a,b,c)$ for some $a,b,c$. Suppose that every two vertices have at most two common edges. If $u,v$ shared two edges, say $e_1$ and $e_3$, then $H\cong B_3^5(a,b,c)$ when $\{u,w\}\subset e_4$, and  $H\cong B_3^6(a,b,c)$ when $\{v,w\}\subset e_4$. It is similar when $u,w$ are shared by two edges. If $v,w$ share two edges, then $e_3\cap e_4=\{v,w\}$ and $H\cong B_3^6(a,b,c)$ for some $a,b,c$.

Therefore, $k$-graphs in $\mathbb{B}_3^{m}$ have six forms $B_3^j(a,b,c)$, $j=1,\cdots,6$, where $a,b,c\in \mathbb{N}$ and $c\ge 1$ for $j=2,4$.

%Then there are exactly four non-pendent edges in $H$ all consisting $k-2$ pendent vertices and $2$ non-pendent vertices. Otherwise $H$ is acyclic, unicyclic or

\begin{Lemma}\label{b3}
Let $H$ be a $k$-graph in $\mathbb{B}_3^{m}\backslash \{B_3^1(m-2,0,0)\}$. If $m\ge 5$, then
$$ \rho(H) < \rho(B_2(m-4,1)) < \rho(B_3^1(m-2,0,0)). $$
\end{Lemma}

\begin{proof}
To prove the first inequality, we discuss by the formation of $H$.

{\bf Case 1.} $H\cong B_3^1(a,b,c)$.
Since $H\ncong B_3^1(m-2,0,0)$,  at least two of $a,b,c$ are positive, say $a,b$.

If $c=0$, then by moving $b-1$ pendent edges from $v$ to $u$, or by moving $a-1$ pendent edges from $u$ to $v$, we obtain $B_3^1(m-3,1,0)$.
Thus by Lemmas~\ref{e} and \ref{bgraph}, $$ \rho(H) \le\rho(B_3^1(m-3,1,0)) < \rho(B_2(m-4,1)). $$

If $c\ge 1$, then by Corollary~\ref{pm}, take $u_1=v$ and $u_2=w$, we have that
$$ \rho(H) < \rho(B_3^1(a,b+c,0))\le\rho(B_3^1(m-3,1,0)) < \rho(B_2(m-4,1)). $$

{\bf Case 2.} $H\cong B_3^3(a,b,c)$.

Suppose that $a\ge c$ within this case. If $a=c=0$, then by Lemma~\ref{bgraph}, $\rho(H) =\rho(B_3^3(0,m-3,0)) < \rho(B_2(m-4,1))$. If $a\ge 1$, then by Corollary~\ref{pm},
$$\rho(H)<\rho(B_3^3(0,a+b,c))\le \rho(B_3^3(0,a+b+c,0))< \rho(B_2(m-4,1)). $$

{\bf Case 3.} $H\cong B_3^2(a,b,c)$ with $c\ge 1$.

Suppose that $a\ge b$ within this case. If $a\ge 1$, then by Corollary~\ref{pm}, take $u_1=v$ and $u_2=w$, we have for $b+c\ge 1$ that
$$ \rho(H) < \rho(B_2(a,b+c)) \le \rho(B_2(m-4,1)). $$

Suppose that $a=b=0$ (See Figure~\ref{b}). By removing one pendent edge from $w$ to $u$, we obtain $B_3^2(1,0,m-4)$. Besides, by removing a non-pendent edge not containing $w$ from $u$ to $w$, we obtain $B_3^3(0,0,m-3)$ from $H$.  Then by Lemma~\ref{e} and the discussion in Case $2,3$, $$\rho(H) <\max\{\rho(B_3^2(1,0,m-4)),\rho(B_3^3(0,0 ,m-3))\} < \rho(B_2(m-4,1)).$$

{\bf Case 4.} $H\cong B_3^j(a,b,c)$, $j=4,5,6$.

If $H\cong B_3^4(a,b,c)$ with $c\ge 1$, then by moving $c$ pendent edges from $w$ to an arbitrary pendent vertex in an edge containing $u,v$, we obtain $B_3^2(a+1,b,c)$. By Lemma~\ref{ne} and the discussion in Case 3, we have for $c\ge 1$ that
$$\rho(H)<\rho(B_3^2(a+1,b,c))<\rho(B_2(m-4,1)).$$

If $H\cong B_3^5(a,b,c)$ , then by moving $c$ pendent edges and one edge containing $u,w$ from $w$ to an arbitrary pendent vertex in an edge containing $u,v$, we obtain $B_3^3(b, a+1,c)$. By Lemma~\ref{ne} and the discussion in Case 2,
$$\rho(H)<\rho(B_3^3( b,a+1,c))<\rho(B_2(m-4,1)).$$

If $H\cong  B_3^6(a,b,c) $, then by moving $c$ pendent edges and the edge containing $v,w$ from $w$ to an arbitrary pendent vertex in an edge containing $u,v$, we obtain $B_3^3( a+1,b,c)$. By Lemma~\ref{ne} and the discussion in Case 2,
$$\rho(H)<\rho(B_3^3(a+1,b,c))<\rho(B_2(m-4,1)).$$

The second inequality of this lemma follows from Lemmas~\ref{bgraph} and \ref{b2}.
\end{proof}

\begin{Lemma}\label{bi}
Let $i\ge 4$ and $H$ be a $k$-graph in $\mathbb{B}_i^{m}$. Then
$$\rho(H)< \max\{\rho(F): \text{$F$ is a $k$-graph in~} \mathbb{B}_{i-1}^{m}\}.$$
\end{Lemma}

\begin{proof}
If all non-pendent vertices in $H$ are in one edge say $f$, then we can find two non-pendent vertices $v_1,v_2$ that do not have other common edge. Otherwise every two non-pendent vertices shares exactly two common edges, then $H$ contains a $3$-cyclic subgraph which is a $k$-graph obtained from $B_3^3(0,0,0)$ by adding an edge containing $u,w$ and $k-2$ new pendent vertices, a contradiction.
Denote by $H'$ the $k$-graph obtained from $H$ by moving all edges incident with $v_2$ except $f$ from $v_2$ to $v_1$. Note that $H'\in \mathbb{B}_{i-1}^{m}$. Now by Corollary~\ref{pm} we have $\rho(H)<\rho(H').$

Suppose there exists two non-pendent vertices $v_1,v_2$ in $H$ that do not have common edge. Let $P=v_1e_1\cdots e_sv_2$ be a shortest path connecting $v_1$ and $v_2$ where $s\ge 2$. Let $H_1$ be the $k$-graph obtained from $H$ by moving all edges incident with $v_1$ except $e_1$ from $v_1$ to $v_2$.  Let $H_2$ be the $k$-graph obtained from $H$ by moving all edges incident with $v_2$ except $e_s$ from $v_2$ to $v_1$. Then $H_1,H_2$ are in $\mathbb{B}_{i-1}^{m}$ and by Lemma~\ref{e} $\rho(H)<\max\{\rho(H_1),\rho(H_2)\}.$

Therefore, $\rho(H)$ is bounded up by the maximum spectral radius among $k$-graphs in $ \mathbb{B}_{i-1}^{m}$ when $i\ge 4$.
\end{proof}

Finally we consider that $H$ is in $\mathbb{B}_4^{m}$.

\begin{Lemma}\label{b4}
Let $H$ be a $k$-graph in $\mathbb{B}_4^{m}$ with $m\ge 5$. Then $$\rho(H)<\rho (B_2(m-4,1)).$$
\end{Lemma}
\begin{proof}
{\bf Case 1.}
$H$ has exactly two non-pendent edges, say $e,f$.

Then $|e\cap f|=3$, otherwise $H$ is acyclic, unicyclic or has a $3$-cyclic subgraph. Hence $H$ can be obtained from $B_3^1(a,b,c)$ by attaching $d$ pendent edges at an arbitrary pendent vertex $t$ in a non-pendent edge, where $d\ge 1$.

Suppose that $a\ge b\ge c$. If $a\ge 1$, then by Corollary~\ref{pm} and Lemma~\ref{b3}, $\rho(H)<\rho(B_3^1(a,b+d,c))<\rho (B_2(m-4,1))$ for $b+d\ge1$. If $a=b=c=0$, then $H\cong B_4$. By Lemma~\ref{bgraph}, $\rho(H)=\rho(B_4)<\rho (B_2(m-4,1))$.

{\bf Case 2.} $H$ has at least three non-pendent edges.

{\bf Subcase 2.1} All non-pendent vertices of $H$ are in one edge, say $f$.
%If the incident edges of $v_1$ except $f$ are pendent edges, then by moving all pendent edges from $v_1$ to $v_2$ we obtain the hypergraph $H'$ from $H$ where $H'\in \mathbb{B}_3^{m,k}$ and $\rho(H)<\rho(H')$  by Lemma~\ref{pm}. Since $H'$ has at least three non-pendent edges, $H'\ncong \{B_3^1(m-2,0,0)\}$. By Lemma~\ref{b3}, $\rho(H)<\rho(H')<\rho(B_2(m-4,1)).$
%
%Suppose that each non-pendent vertex of $H$ is contained in at least two
According to the discussion within the proof of Lemma~\ref{bi}, we can find two non-pendent vertices $v_1$ and $v_2$ that do not share any edge other than $f$.  Moving all edges incident with $v_2$ except $f$ from $v_2$ to $v_1$, we obtain from $H$ a $k$-graph $H'\in \mathbb{B}_{i-1}^{m}$ which has the same number of non-pendent edges as $H$ does. Then $H'$ has at least three non-pendent edges, and thus $H'\ncong B_3^1(m-2,0,0)$.
By Corollary~\ref{pm} and Lemma~\ref{b3},   we have $\rho(H)<\rho(H')<\rho (B_2(m-4,1)).$

{\bf Subcase 2.2} There exists two non-pendent vertices $v_1,v_2$ in $H$ that do not have any common edges.

Let $P=v_1e_1\cdots e_sv_2$ be a shortest path connecting $v_1$ and $v_2$ where $s\ge 2$. Denote by $H_1$ the $k$-graph obtained from $H$ by moving all edges incident with $v_1$ except $e_1$ from $v_1$ to $v_2$.  Denote by $H_2$ the $k$-graph obtained from $H$ by moving all edges incident with $v_2$ except $e_s$ from $v_2$ to $v_1$. Obviously $H_1$ and $H_2$ are in $\mathbb{B}_3^{m}$.
 Next we prove that they are not $B_3^1(m-2,0,0)$.

 If there is a pendent edge attaching at $v_1$ or $v_2$ in $H$, then in $H_1$, the shortest path connecting $v_1$ and an arbitrary pendent vertex in a pendent edge attached at $v_2$ is of length at least $3$. This implies that $H_1\ncong B_3^1(m-2,0,0)$, as the maximum length over all paths in $B_3^1(m-2,0,0)$ is $2$. Similarly we have $H_2\ncong B_3^1(m-2,0,0)$.

Suppose that $v_1$ and $v_2$ are not in any pendent edge. Then each of $v_1,v_2$ is incident with a non-pendent edge other than $e_1$ and $e_s$, say $f_1$ and $f_2$ respectively. Thus there are three edges $(f_1\backslash \{v_1\})\cup \{v_2\}$, $f_2$ and $e_s$  being non-pendent  in $H_1$, which implies that $H_1$ is distinct with $B_3^1(m-2,0,0)$ where only two edges being non-pendent. Similarly we have $H_2\ncong B_3^1(m-2,0,0)$.

Thus by Lemmas~\ref{e} and \ref{b3},
$$\rho(H)<\max\{\rho(H_1),\rho(H_2)\}<\rho (B_2(m-4,1)).$$
%If $H$ is linear, then according to \cite[Theorem 3.9]{FTPL15} and \cite[Theorem 8 and 10]{KLQY16}, $\rho(H)<\rho(B_m^L(1))$ (See Fig.~\ref{bg}). Observe that from $B_m^L(1)$, by moving $e_1$ from $v$ to $t$ or by moving $e_2$ from $t$ to $v$, we can obtain $B_3^6(m-4,0,0)$. Thus $\rho(H)<\rho(B_m^L(1))<\rho(B_3^6(m-4,0,0))<\rho (B_2(m-4,1))$.
%
%Suppose that $H$ is non-linear.

Now the proof is completed.
\end{proof}

By Lemmas~\ref{bi} and \ref{b4}, we have:

\begin{Lemma}\label{b5}
Let $H$ be a $k$-graph in $ \mathbb{B}_i^{m}$, where $i\ge 4$ and $m\ge 5$. Then
$$\rho(H)<\rho (B_2(m-4,1)).$$
\end{Lemma}

\begin{Theorem}\label{bg}
Let $H$ be a $k$-graph in $ \mathbb{B}^{m}$ with $m\ge 5$. Then

$(i)$
$~\rho(B_2(m-4,1))<\rho(B_2(m-3,0))=\rho(B_3^1(m-2,0,0));$

$(ii)$ if $H \notin \{B_2(m-4,1),B_2(m-3,0),B_3^1(m-2,0,0)\}$,
then
$$\rho(H)<\rho(B_2(m-4,1)).$$
\end{Theorem}

\begin{proof}
The relation in $(i)$ follows directly from Lemmas~\ref{bgraph} and \ref{b2}.

The inequality of $(ii)$ can be obtained from Lemmas~\ref{b2}, \ref{b3} and \ref{b5} by specifying the number of non-pendent vertices in $H$.

Then the proof is completed.
\end{proof}

%\section{Characterization of the maximizing graphs among $r$-cyclic hypergraphs}

\end{document}